\documentclass[11pt]{amsart}
\usepackage{latexsym}
\usepackage{graphics}
\usepackage{graphicx}
\usepackage{epstopdf}
\usepackage{tikz}

\usepackage{caption}
\usepackage[labelformat=simple]{subcaption}

\newcommand{\Vol}{\mbox{\rm Vol}}

\newtheorem{theorem}{Theorem}[section]
\newtheorem{proposition}[theorem]{Proposition}

\theoremstyle{remark}
\newtheorem{remark}[theorem]{Remark}
\newtheorem{example}[theorem]{Example}

\title[Refined upper bound in terms of the colored Jones polynomial]{A refined upper bound for the hyperbolic volume of alternating links 
and the colored Jones polynomial}
\author{Oliver Dasbach \and Anastasiia Tsvietkova}
\date{}

\begin{document}

 \footnotesize
 \begin{abstract} We give a refined upper bound for the hyperbolic volume of an alternating link in terms of the first three and the last three coefficients of 
 its colored Jones polynomial.
 \end{abstract}
 
\maketitle
\normalsize

\section{Introduction}
 
Since quantum invariants were introduced into knot theory, there has been a strong interest in relating them to the intrinsic geometry of a 
link complement. This is for example reflected in the Volume Conjecture \cite{Kashaev:VolumeConjecture, Murakami:VolumeConjecture}, which claims that the hyperbolic volume of a 
link complement in $S^3$ is determined by the colored Jones polynomial. The conjecture has only been proven for very few hyperbolic knots and links.  
 
In this note, we are concerned with correlating the volume with coefficients of the colored Jones polynomial of a hyperbolic alternating link. 
To compute the volume of a particular link complement one may use, for example, techniques developed in SnapPea \cite{Weeks:SnapPea} or in \cite{TT:AlternativeApproachToVolume,Tsvietkova:2Bridge}. However, 
there is no simple expression known for the volume function in general, and one way to establish such a correlation is to find bounds for the volume, and relate those bounds to the colored Jones polynomial \cite{DL:VolumeIsh,FKP:Guts}.

In \cite{Adams:Thesis} and \cite{Lackenby:Volume}, an upper bound for the volume of a hyperbolic link in terms of the number of crossings of a link diagram is obtained. In \cite{Lackenby:Volume}, the bound is also restated in terms of the number of twists of a diagram, and is improved in the appendix so that the result is sharp within its framework. The latter bound is given by a beautifully simple expression that applies to all hyperbolic links. Naturally, despite being sharp in general, this bound is often a serious 
overestimate. The aim of this paper is to give a refined upper bound that uses both the number of twists and the number of crossings within a twist:

 \theoremstyle{theorem}
 \newtheorem*{reptheorem}{Theorem \ref{Thm Upper Bound}}
 \begin{reptheorem}
 Given a diagram $D$ of a hyperbolic alternating link $K$, with $t_i$ twist regions of precisely $i$ half-twists, and $g_i$ twist regions of at least $i$ half-twists, then:
 $$ Vol(S^3- K) \leq (10 g_4(D)+ 8t_3(D)+6t_2(D)+4 t_1(D) - a ) v_3, $$
 where $a =10$ if $g_4$ is non-zero, $a=7$ if $t_3$ is non-zero, and $a=6$ otherwise. 
 \end{reptheorem}

The bound provides better estimates for links with smaller volumes, and thus is better suited to our purposes. This result allows us to bound the volume in terms of coefficients of the colored Jones polynomial:

\newtheorem*{reptheorem1}{Theorem \ref{refinedVolumeIsh}}
\begin{reptheorem1}
Let $K$ be an alternating, prime, non-torus link,
and let 
\begin{equation*} 
J_K(n) = \pm (a_n q^{k_n} - b_{n} q^{k_n-1} + c_{n} q^{k_n-2})+\ldots \pm (\gamma_n q^{k_n-r_n+2} - \beta_n q^{k_n-r_n+1} + \alpha_n q^{k_n-r_n})
\end{equation*}
 be the colored Jones polynomial of $K$, where $a_n$ and $\alpha_n$ are positive.  Then
$$\Vol(S^3-K) \leq \left ( 6 ((c_2+\gamma_2) - (c_3 + \gamma_3)) - 2 (b_2+ \beta_2) -a \right ) v_3 \leq 10 (b_2+\beta_2-1) v_3,$$
where $a=10$ if $b_2+\beta_2 \neq (c_2-c_3)+(\gamma_2-\gamma_3)$ and $a=4$ otherwise.
\end{reptheorem1}

Note that this theorem utilizes the third and the prepenultimate coefficients of the colored Jones polynomial, while previously only the correlation of the first two and last two coefficients with the volume was shown.

\section{A refined upper bound for the hyperbolic volume}
 
Given a reduced alternating diagram $D$ of a hyperbolic link, we will denote the number of crossings by $c(D)$. The crossing number of a link is the minimal crossing number among all its diagrams, and is denoted by $c$. A bigon is a region of the diagram having exactly two crossings in its boundary. A twist is either a connected collection of bigons arranged in a row (Figure \ref{twistloop1}), which is not a part
of a longer row of bigons, or a single crossing adjacent to no bigons. The twist number $t(D)$ is the number of twists in the diagram $D$. Furthermore, denote the number of twists that have exactly $i$ crossings (i.e. $i-1$ bigons) by $t_i(D)$, and denote the number of twists that have at least $i$ crossings by $g_i(D)$. In particular, $g_1(D)=t(D)$ and $g_2(D)=t(D)-t_1(D)$.

\begin{figure}
\centering
\begin{subfigure}[b]{0.4 \textwidth}
\centering
\includegraphics[scale=0.27]{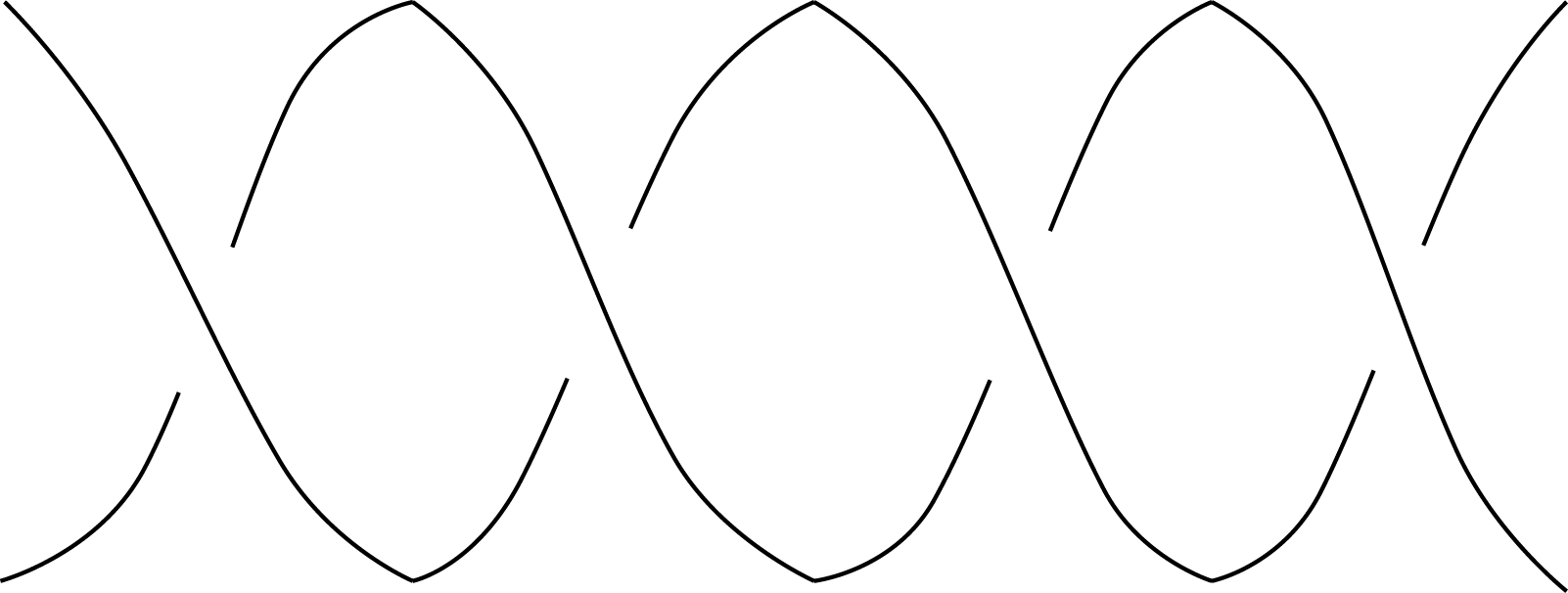}
\caption{A twist with $4$ crossings}
\label{twistloop1}
\end{subfigure}
\begin{subfigure}[b]{0.4 \textwidth}
\centering
\includegraphics[scale=0.27]{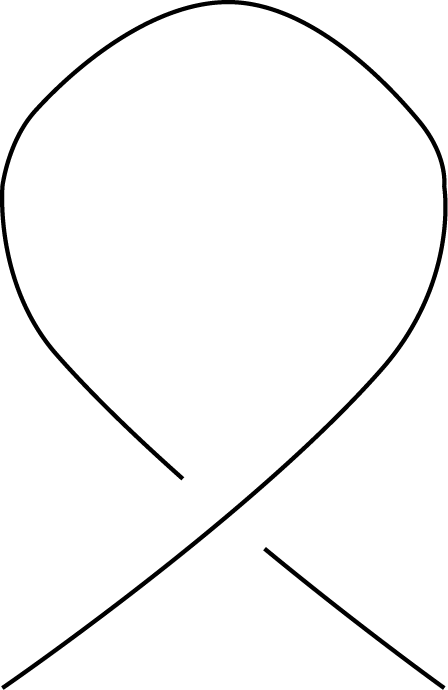} 
\caption{A loop}
\label{twistloop2}
\end{subfigure}
\caption{Twists and loops}
\end{figure}

In this section, we modify the approach suggested in \cite{Lackenby:Volume}, in order to obtain an upper bound for the volume in terms of $t_i(D)$ and $g_i(D)$. This includes two steps. Firstly, we describe a decomposition of an augmented link into two ideal polyhedra. 
Secondly, we perform a triangulation of the polyhedra and count the number of ideal tetrahedra. In the next section, we use this to prove a result reflecting the relation of the volume with the first three and last three coefficients of the colored Jones polynomial. 
 
We will start with a review of previously obtained upper bounds for hyperbolic volume of links. Henceforth, let $v_3$ be the volume of a regular ideal hyperbolic tetrahedron (therefore, $v_3$ is the greatest possible volume of an ideal hyperbolic 3-simplex). 
The first upper bound, for knots and in terms of the number of crossings, appears in \cite{Adams:Thesis} (see \cite{CallahanReid:Hyperbolic}).  

\begin{theorem}[C. Adams, 1983] For a hyperbolic knot $K$ with crossing number $c$, different from the figure-eight
knot, $$\Vol(S^3 - K) \leq (4c - 16) v_3.$$
\end{theorem}

This result is recovered for hyperbolic links by Marc Lackenby \cite{Lackenby:Volume} and it is reformulated in terms of the number of twists. In an appendix it is improved by Ian Agol and Dylan Thurston as follows:

\begin{theorem}[\cite{Lackenby:Volume}] \label{Lackenby:Volume}
Given a diagram $D$ of a hyperbolic link $K,$
$$\Vol (S^3 - K) \leq 10 (t(D) - 1) v_3.$$
\end{theorem}

We will improve the latter bound for alternating links at a cost of introducing more parameters. Let us start by using Lackenby's technique of modifying a link diagram through the addition of crossing circles. In particular, given a reduced alternating diagram of a link $K$, we encircle every twist that has at least four crossings by a simple closed curve, called a crossing circle. The obtained link $J$  is an augmented alternating link, and therefore is hyperbolic \cite{Adams:Augmented}. Moreover, $S^3 -K$ can be obtained from $S^3 - J$ by Dehn fillings of the tori that correspond to crossing circles, and therefore the volumes satisfy $\Vol(S^3 - K)\leq\Vol(S^3 - J)$ by Theorem 6.5.6 in \cite{Thurston:Book} (the equality holds if and only if none of the twists is augmented, and the link is unchanged). Next we delete all crossings in the twists that we encircled, obtaining a new link $L$. The volume of $S^3 - L$  is equal to the volume of $S^3 - J$ by \cite{Adams:3punctured}. Note that although the diagram of $J$ is not alternating, it is reduced in the sense that there are no loops as shown in Figure \ref{twistloop2}.

Denote the diagram of $L$ by $D_a$ (subindex $a$ stands for ``augmented"). Note that $D_a$ consists of 2-tangles (i.e. tangles from which there emerge four arcs pointing out) $T_1, \ldots, T_p$, connected with each other or with itself through combinations of crossing circles (as, for example, in Figure \ref{tangles1}). We may assume that none of $T_k, \, k=1, \ldots p,$ is trivial in the sense of Lickorish \cite{Lickorish:Tangles}, i.e. none consists of two parallel strands only.

\begin{figure}
\begin{center}
\begin{subfigure}[b]{0.4 \textwidth}
\centering
\includegraphics[scale=0.53]{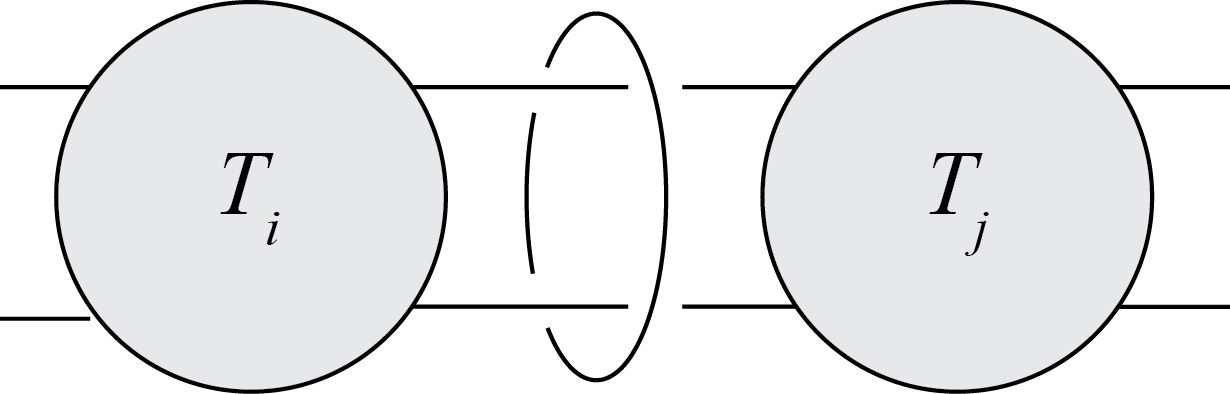}
\subcaption{Non-trivial tangles \label{tangles1}}
\end{subfigure}
\begin{subfigure}[b]{0.4 \textwidth}
\centering
\includegraphics[scale=0.33]{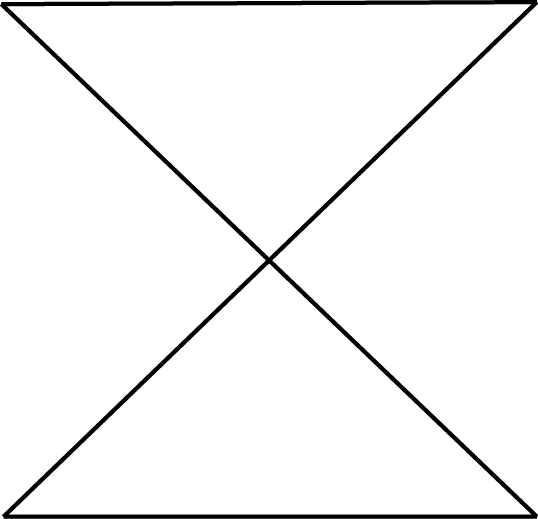} 
\subcaption{``Bow-tie'' faces \label{tangles2}}
\end{subfigure}
\end{center}
\caption{Tangles \label{tangles}}
\end{figure}

The diagram $D_a$ lies in the projection plane (denote the plane by $P$) except for small perturbations at crossings and the crossing circles. In \cite{Menasco:PolyhedraDecomposition}, a technique of decomposing a link complement into two ideal polyhedra, one above $P$, and one below, is described. In particular, a constructive algorithm giving such a decomposition for alternating links is given. In the next few paragraphs, we somewhat recall and modify the technique.

Imagine thickening the strands of $D_a$ until they (almost) touch each other at crossings. Each crossing circle pierces $P$ twice, and is divided into two half-circles by $P$. The two-punctured disk that it bounds (called a twisting disk) is divided by $P$ into two half-disks. The projection plane itself is divided into planar regions $R_1, \ldots, R_m$ by thickened strands of a link and by the two-punctured disks that ``cut" into it. More precisely, we say that a planar region $R_j$ is bounded by the segments in $P$ that lie either on the boundary of a thickened strand, or at the intersection of the interior of a two-punctured disk with $P$. A segment ends whenever it meets another strand, a crossing circle or a twisting disk.  Figure \ref{planarregions} shows examples of a planar $4$-sided region and a planar $3$-sided region in the projection plane respectively.

\begin{figure}[ht]
\centering
\begin{subfigure}[b]{0.4 \textwidth}
\centering
\includegraphics[scale=0.36]{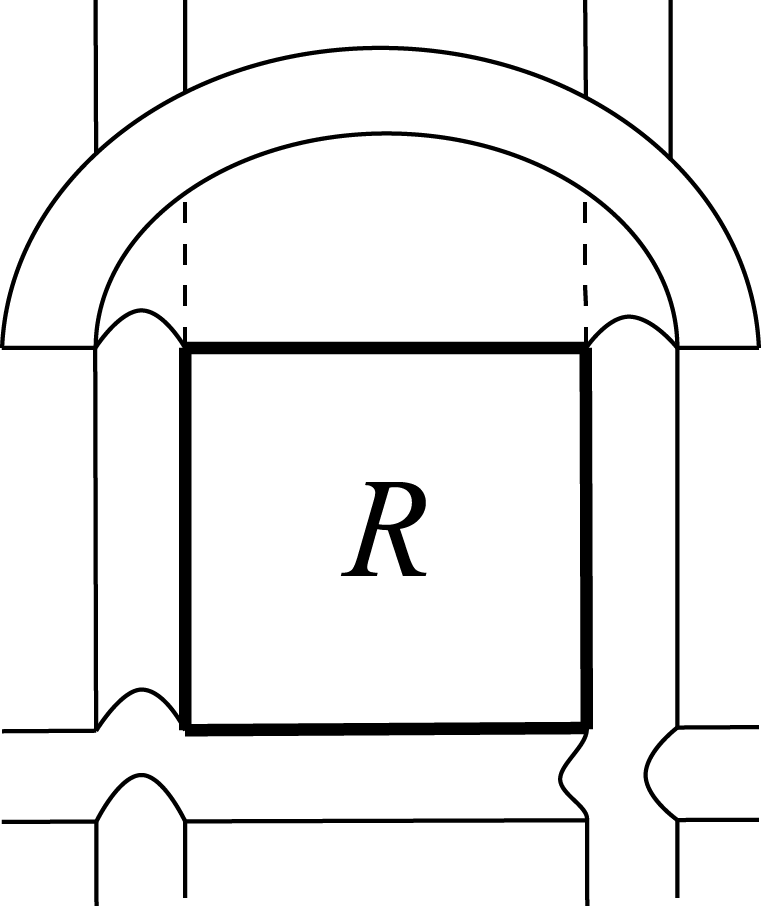}
\caption{A planar $4$-sided region $R$}
\label{planarregions1}
\end{subfigure}
\begin{subfigure}[b]{0.4 \textwidth}
\centering
\includegraphics[scale=0.38]{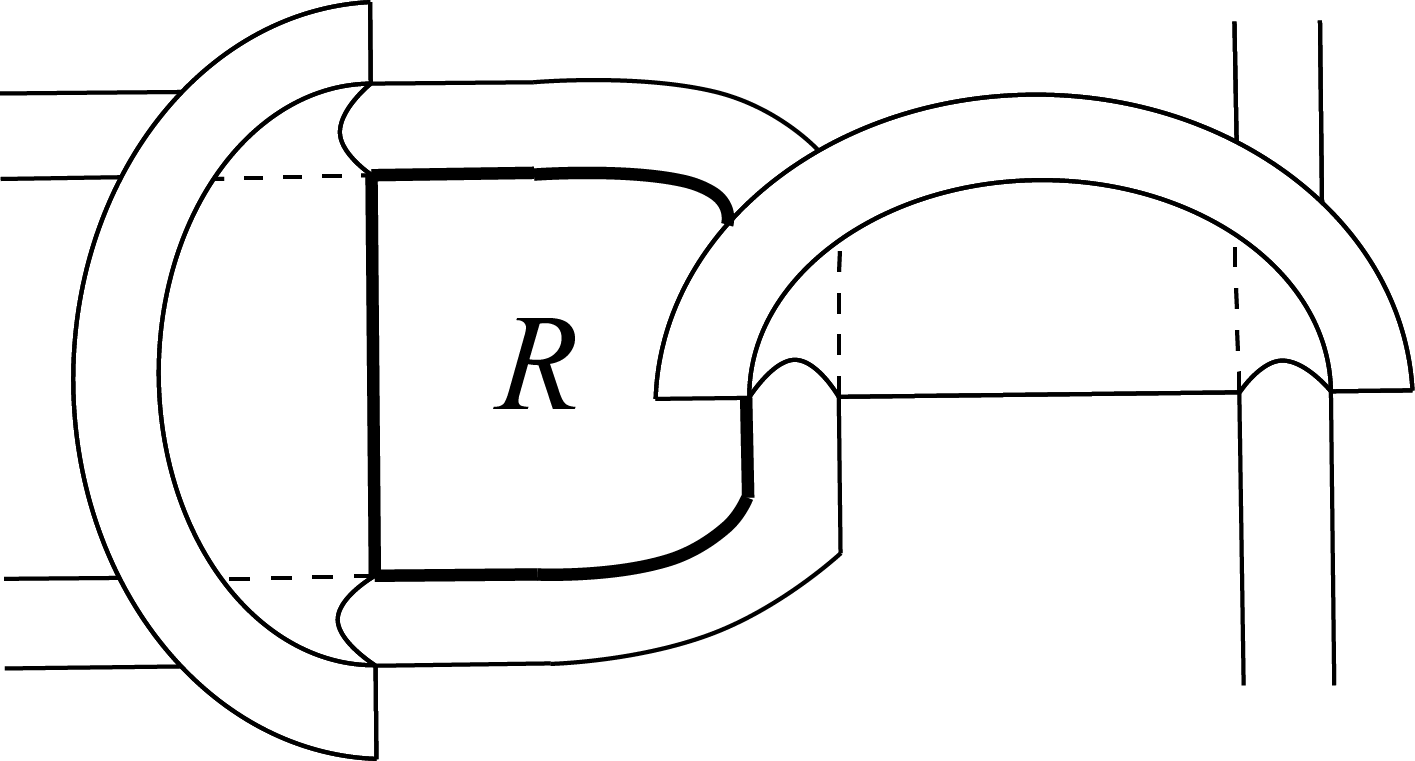}
\caption{A planar $3$-sided region $R$}
\label{planarregions2}
\end{subfigure}
\caption{Planar regions $R$}
\label{planarregions}
\end{figure}

Let us put two balls $B_1$ and $B_2$ in $S^3 - L$, one above the projection plane, and one below. We inflate both balls simultaneously so that for an observer situated at the plane, the balls look like mirror images of each other. Eventually, the balls fill all of the $S^3 - L$ and touch themselves and each other in the following way: (i) an upper ball touches itself at the upper half of every twisting disk; (ii) similarly, the lower ball touches itself at the lower half of every twisting disk; (iii) the upper ball touches the lower ball at every planar region $R_j, \, j=1, \ldots, m,$ lying in the projection plane.

On the boundary of each $B_i, i=1$ or $i=2$, there is a planar graph $\Gamma_i$ determined by the diagram $D_a$ as follows:
 Consider all parts of strands of the link that are visible from above of $P$. Subdivide them into segments $s_1, s_2, \ldots, s_h$ as described above.  Each vertex of $\Gamma_1$ corresponds to one of $s_i, \, i=1, 2, \ldots, h$, and vice versa. The edges of $\Gamma_1$ arise where the thickened strands of $D_a$ ``touch" each other, i.e. at crossings or crossing circles, in the following way. Where $B_1$ touches one of the half-disks from both sides, there are two regions in $\Gamma_1$, with three vertices each, connected by edges as a bow-tie, as in 
Figure \ref{tangles2} (these faces are described in the appendix to \cite{Lackenby:Volume}).  In addition, every crossing of $T_k, \, k=1, 2,  \ldots p,$ yields four new edges of $\Gamma_1$. The four arcs in $S^3 - L$, corresponding to these edges, travel from an overpass to an underpass of this crossing in $D_a$. The arcs are all homotopic to one another, as described in \cite{Menasco:PolyhedraDecomposition}. 

We have listed all edges and vertices of $\Gamma_1$ by now. Now two edges of $\Gamma_1$ meet at the same vertex if and only if the corresponding arcs in the link complement meet the same $s_i$. The construction implies that every vertex of the graph is now four-valent, with some multiple (to be precise, double) edges resulting from twists of $D_a$ that have at least four crossings. We will identify all the edges between the same two vertices $v_1, v_2$, obtaining just one edge instead. The second graph, $\Gamma_2$, is constructed similarly, but using the segments of the strands of $D_a$ visible from below. Note that the diagram 
$D_a$ from above looks like the mirror image of the diagram from below, but with overcrossings and undercrossings interchanged.

Figure \ref{Gamma} demonstrates an example of a link diagram (the leftmost picture) and the corresponding graph $\Gamma_1$ (the rightmost picture). Two pairs of the bow-tie faces are shaded in the graph, and all segments $s_i$ of $D_a$ are numbered, except for the two upper half-circles. 
On $\Gamma_1$, the numbering of the vertices corresponds to the numbering of $s_i$ in $D_a$. 

\begin{figure}[ht]
\begin{center}
\includegraphics[scale=0.51]{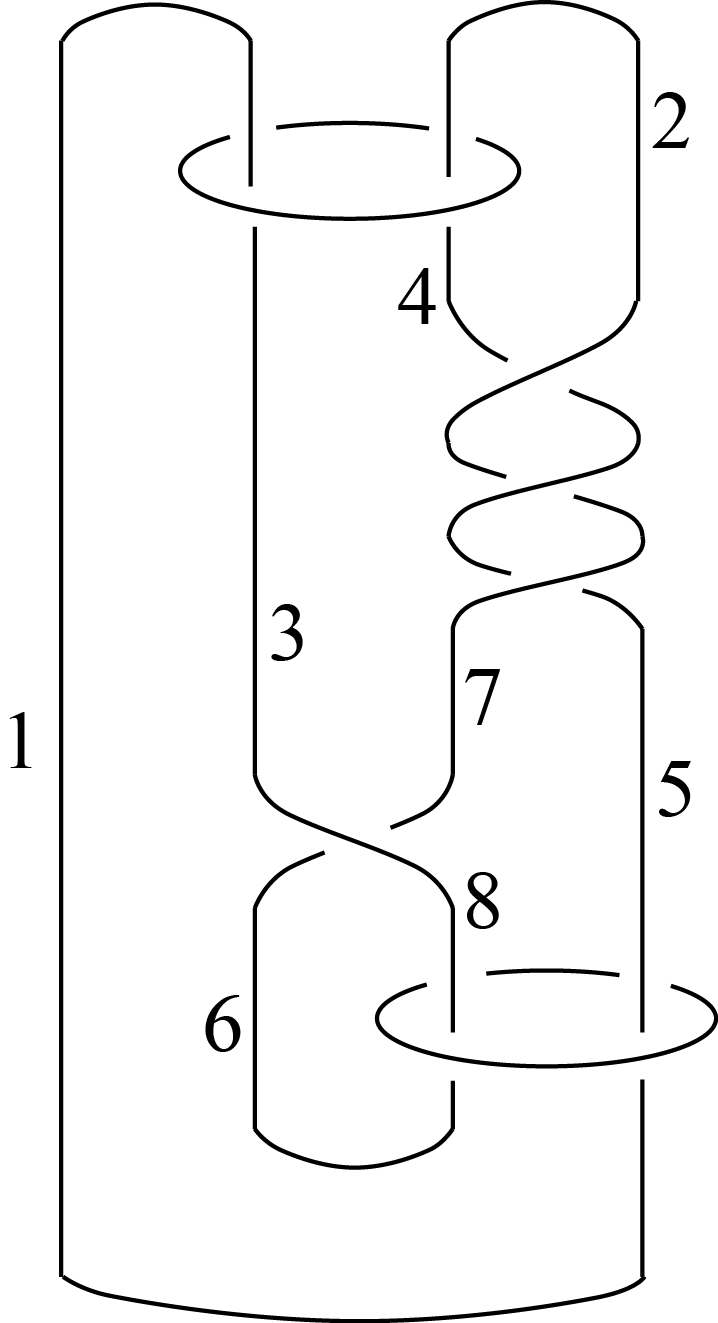} \hspace{0.6in}
\includegraphics[scale=0.51]{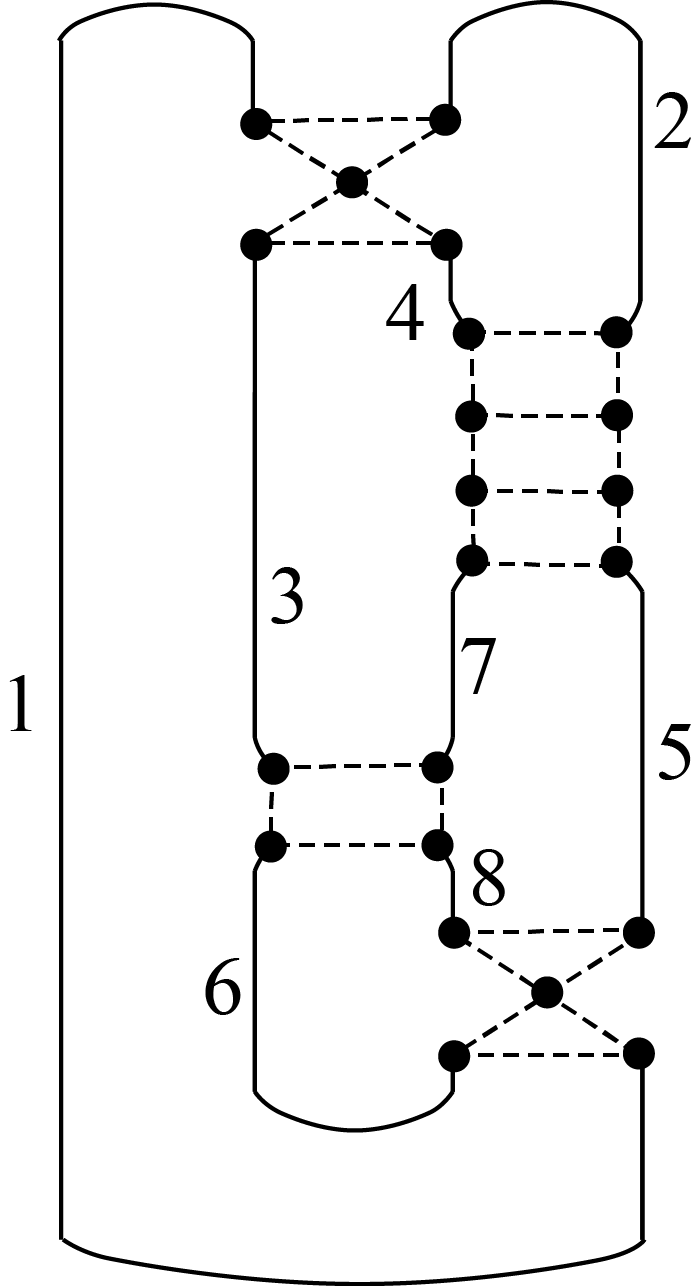} \hspace{0.5in}
\includegraphics[scale=0.58]{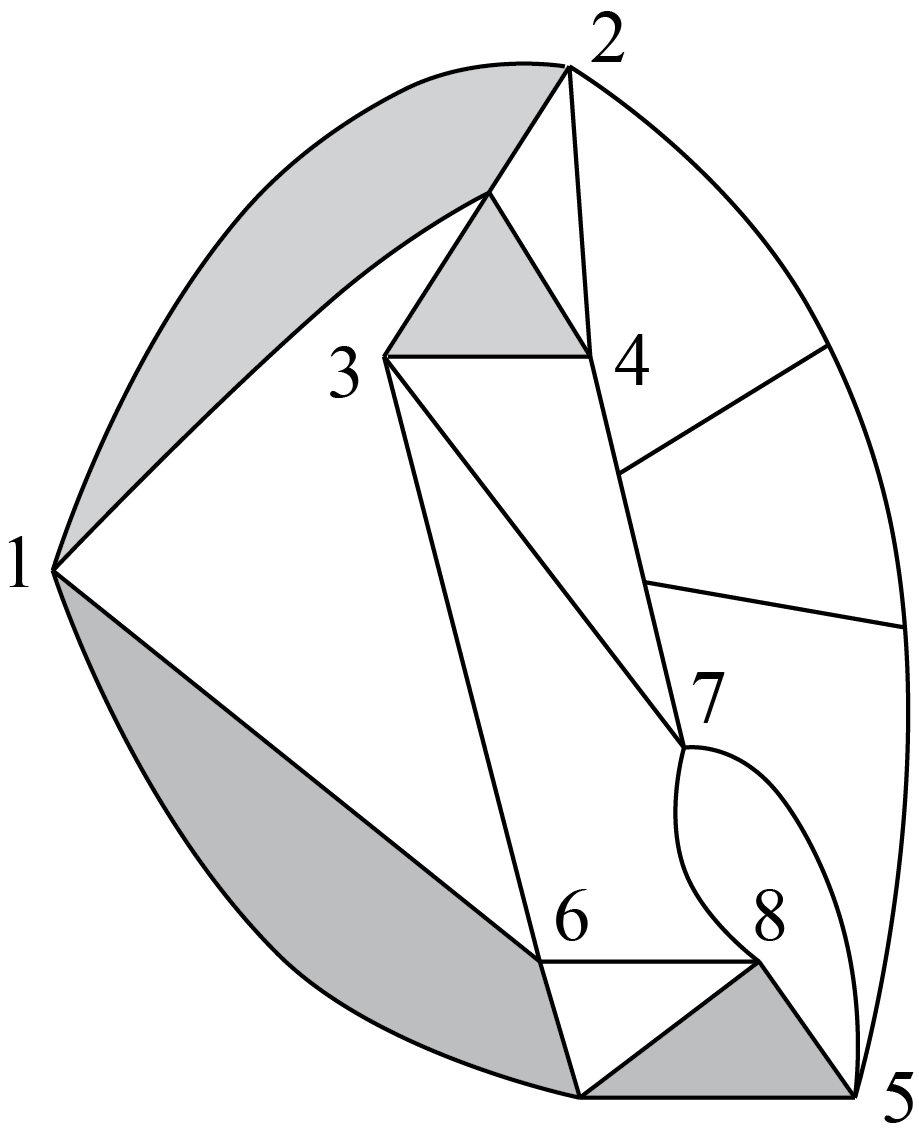}
\end{center}
 \caption{A diagram $D_a$, the auxiliary picture, and the corresponding graph $\Gamma_1$ \label{Gamma}}
\end{figure}

The above description is somewhat implicit. To make it more instructional, we introduce an intermediate auxiliary picture.

As a first step, substitute all crossing circles in $D_a$ by bow-ties. Then substitute every crossing that is left by a 4-sided face, whose vertices lie on the edges of $D_a$ coming from this crossing (we will call such faces ``pillowcases"). Whenever two vertices happen to be on one segment $s_i$ of $D_a$, merge them into one vertex. First we do this for the vertices of the pillowcase faces (we obtain a picture shown in Figure \ref{Gamma}, in the middle, where edges of the new bow-tie and pillowcase faces are depicted by dotted lines, and the vertices shown as black dots). Then we perform merging for the rest of the vertices.  We obtain the graph $\Gamma_1$ (Figure \ref{Gamma}, right). 

The balls $B_1$ and $B_2$  together with $\Gamma_i$ form polyhedra decomposing $S^3 - L$. We will adopt the notation $P_1$ and $P_2$ for them. The link complement is obtained by folding the bow-ties in $P_i$ along each vertex to glue the pairs of triangles together; then folding each pillowcase so that two of its vertices that were on the same overpass are glued together; then, for each folded pillowcase located at a crossing (one pillowcase from $\Gamma_1$, and one from $\Gamma_2$), identifying the four edges (two from each pillowcase) as in \cite{Menasco:PolyhedraDecomposition}; then doubling along the rest of the faces. Figure \ref{pillowcase1} illustrates what happens with a pillowcase in the process: its vertices are labeled by $a, b, c, d$. In particular, for the link diagram in Figure \ref{Gamma}, in order to fold two pillowcases with vertices 3, 6, 7, 8 in $P_1$ and $P_2$, we first identify together 3 and 8 for $P_1$, and 6 and 7 for $P_2$. We then identify all four remaining pillowcase edges (two in $P_1$ and two in $P_2$) as in Menasco.
\begin{figure}[ht]
\centering
\begin{subfigure}[b]{0.6 \textwidth}
\centering
\includegraphics[scale=0.7]{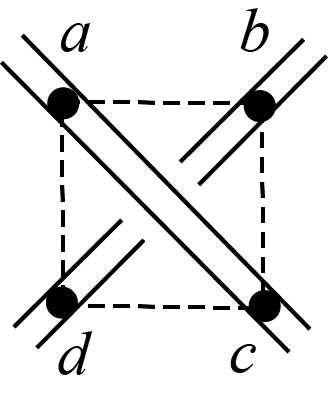} \qquad
\includegraphics[scale=0.7]{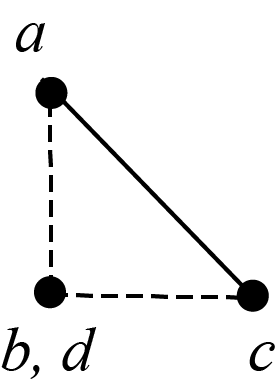} \qquad
\includegraphics[scale=0.7]{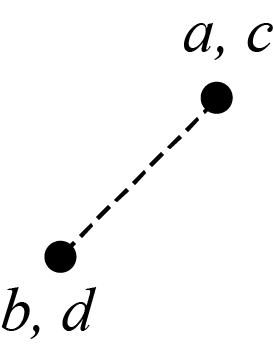}
\caption{Folding a\\ pillowcase}
\label{pillowcase1}
\end{subfigure}
\begin{subfigure}[b]{0.3 \textwidth}
\centering
\includegraphics[scale=0.7]{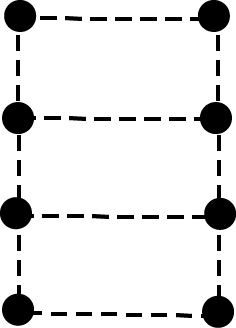}
\caption{Faces resulting from a 3-crossing twist}
\label{pillowcase2}
\end{subfigure}
 
\caption{Pillowcase}
\end{figure}

In the proof below, we will perform a triangulation of  $S^3 - L$, inspired by the one suggested by Ian Agol and Dylan Thurston in the appendix to \cite{Lackenby:Volume}. If every twist of $D$ has only one crossing, then the following theorem essentially recovers Adams' Theorem 2.1 from \cite{Adams:Thesis} (there is only a difference in the additive constant).

\begin{theorem} \label{Thm Upper Bound}
Given a diagram $D$ of a hyperbolic alternating link $K$, 
$$ Vol(S^3- K) \leq (10 g_4(D)+ 8t_3(D)+6t_2(D)+4 t_1(D) - a ) v_3, $$
where $a =10$ if $g_4$ is non-zero, $a=7$ if $t_3$ is non-zero, and $a=6$ otherwise. 
\end{theorem}

\begin{proof}
Put
two new vertices $v_i, \, i=1, 2$, in the interior of $P_i$. Cone the vertices to the faces
of the polyhedra. Every bow-tie face of $P_i$ becomes a base for two tetrahedra. Every pillowcase face of $P_i$ becomes a base of a pyramid that can be subdivided into two tetrahedra. For all the faces of $P_i$ that are not bow-ties or pillowcases, take a pyramid with the base at the face and the vertex $v_i$. For every two faces of $P_1, P_2$ identified together, from two such (similar) pyramids we obtain a bipyramid with the vertices $v_1, v_2$. This bipyramid can be subdivided into $n$ tetrahedra (where the base faces for two pyramids had $n$ sides each), all of which share an edge from $v_1$ to $v_2$. Substitute the polyhedra by the resulting triangulation. 

 We need to count the maximal possible number of tetrahedra. Every crossing circle yields four tetrahedra based on two bow-tie triangles (two tetrahedra above $P$, and two below), and six more tetrahedra adjacent to the edges of the bow-tie from the outside. Every twist with one crossing yields four new edges of $P_1$ (glued to four similar edges of $P_2$), which gives at most eight tetrahedra, adjacent to two different sides of every edge. However, the four tetrahedra based on the pillowcase faces (two above $P$, and two below) degenerate when we perform folding of pillowcases and identifying their edges as described above. Every twist with two crossings yields two pillowcase faces sharing one common edge, and every crossing with three crossings yields three pillowcase faces sharing edges as in Figure \ref{pillowcase2}. Therefore, the number of tetrahedra that will not degenerate, is 6 for a twist with two crossings, and 8 for a twist with three crossings. We obtain $10 g_4(D)+ 8t_3(D)+6t_2(D)+ 4t_1(D)$ tetrahedra total.  
 
Lastly, collapse the non-ideal vertices $v_1, v_2$ to one chosen vertex that $P_1$ shares with $P_2$ under the gluing. In particular, if $g_4(D)$ is non-zero, collapse them to the center of one of the bow-ties, eliminating ten tetrahedra, as described in the appendix to \cite{Lackenby:Volume}. Otherwise collapse them to any other vertex $v$ such that every edge, incident to it, does not belong to two distinct pillowcases originating from the same twist of $D_a$ (in Figure \ref{pillowcase2}, the suitable vertices would be two on the top, and two in the bottom, but not the rest). 
The vertex $v$ is then incident to two four faces: two pillowcases, and two distinct non-pillowcase faces,  $F_1$ and $F_2$. The pillowcase faces did not contribute to the total number of tetrahedra, since they will be folded later. Both $F_1$ and $F_2$ have least three sides. Therefore, we collapse at least 6 tetrahedra. If $t_3$ is non-zero, then at least one of $F_1, F_2$ has four sides, and we collapse at least 7 tetrahedra.

Therefore, we obtain the triangulation of $S^3-L$ with $10g_4(D)+ 8t_3(D)+6t_2(D)+4 t_1(D)-a$ tetrahedra in it, with $a$ as stated in the theorem. Since the greatest possible hyperbolic volume of an ideal tetrahedron is $v_3$, the hyperbolic volume of $S^3-K$ is at most $(10g_4(D)+ 8t_3(D)+6t_2(D)+4 t_1(D)-a)v_3$. From Thurston's Theorem 6.5.6 \cite{Thurston:Book}, the hyperbolic volume of $S^3-K$ is bounded from above by the hyperbolic volume of $S^3-J$.
\end{proof}

 Note that our polyhedral decomposition is not the decomposition described by Menasco \cite{Menasco:PolyhedraDecomposition}, since we introduced more vertices.

 \begin{remark} \label{AdditionalRemark}
 
 Theorem \ref{Thm Upper Bound} is stated for alternating links. In the proof, we pass from a link $K$ to the (partially) augmented link $J$. Theorem 6.5.6 from \cite{Thurston:Book} and Corollary 5 from \cite{Adams:3punctured} are then used. They both require $J$ and $K$ to be hyperbolic, which is always the case if $K$ is a hyperbolic alternating link (due to \cite{Adams:Augmented}). If one chooses to work with a non-alternating hyperbolic link $K$, the link $J$ might be not hyperbolic anymore. 

 \end{remark}

\begin{remark} \label{remark}
An example in the appendix to \cite{Lackenby:Volume} demonstrates that the $10$ in the bound in Theorem \ref{Thm Upper Bound} is an optimal constant, and cannot be improved within this setting. However, for some classes of links one may further improve the bound by subtracting additional terms. The next subsection gives an example.

\subsection{Triangular regions}  \label{Section Triangular Regions}

Let us introduce another parameter $\Delta$, that arises from the link diagram. In the described triangulation, for each triangular region substitute three tetrahedra by two: one above and one below the projection plane.  Under this modification, $\Delta$ is simply the number of triangular faces of $P_1$, and we may subtract it from our previous count:
\begin{equation} \label{Eq with Delta}
Vol(S^3- K) \leq \left (10 g_4(D)+ 8t_3(D)+6t_2(D)+4 t_1(D) - a - \Delta \right ) v_3. 
\end{equation}

This however changes the number of tetrahedra that we can collapse in the end, and $a$ becomes 8 when $g_4$ is non-zero, $6$ if $t_3$ is non-zero, and $4$ otherwise.

In the diagram $D_a$ of the augmented alternating link, $\Delta$ corresponds to the number of regions of $D_a$ that yield triangular regions of $\Gamma_1$. There are two types of such regions in $D_a$. Firstly, the ones bounded by exactly three segments, say, $s_1, s_2, s_3$, lying in $P$ and, in addition, possibly meeting up to three crossing circles. In Figure \ref{Gamma}, the examples of such regions are bounded by segments numbered $3, 4, 7$ or $7, 5, 8$. Secondly, the regions that are bounded by exactly two segments lying in $P$ and, possibly, by one crossing circle, and have exactly one crossing circle piercing the region. The examples are the regions bounded by strands $2, 4$ or by strands $6, 8$ in Figure \ref{Gamma}. Note the difference between crossing circles bounding the regions and piercing the regions: by ``bound" we mean the situation seen in Figure \ref{planarregions1}, and by ``pierce" the situation in Figure \ref{augmentation1}, with respect to the region $R$.

In the diagram $D$, the parameter $\Delta$ corresponds to the number of regions that either are adjacent to exactly three crossings from three different twists, or are adjacent to exactly one twist and only one more crossing that does not belong to this twist, as region $R$ in Figure \ref{augmentation1}, right. E.g., all the triangles for the standard diagram of the figure-eight knot, as in Figure \ref{augmentation2}, are of the latter type. Since this number is uniquely determined by the diagram, we may adopt the notation $\Delta(D)$ instead of just $\Delta$ for further use.

\begin{figure}[ht]
\centering
\begin{subfigure}[b]{0.6 \textwidth}
\centering
\includegraphics[scale=0.5]{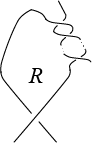} \hspace{0.1in}
\includegraphics[scale=0.5]{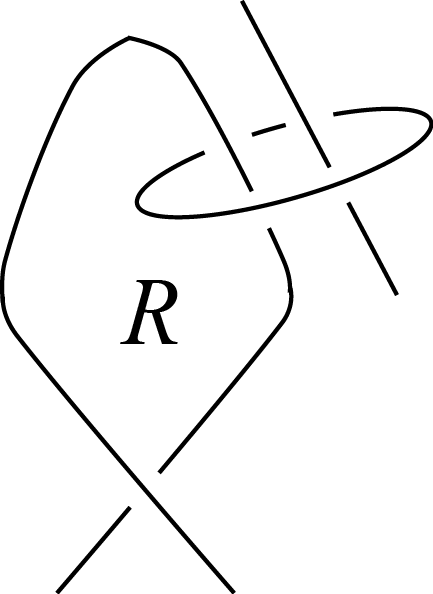}
 \qquad
 \caption{$R$ before and after augmentation}
\label{augmentation1}
\end{subfigure}
\begin{subfigure}[b]{0.3 \textwidth}
\centering
\includegraphics[scale=0.3]{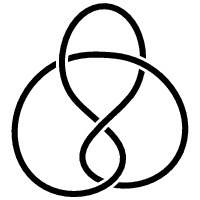}
\caption{The figure-eight knot} 
\label{augmentation2}
\end{subfigure}
 
\caption{Augmentation}
\end{figure}

\end{remark}

\begin{example}
Consider the reduced alternating diagram of the figure-8 knot as in Figure \ref{augmentation2}. Recall that the actual hyperbolic volume of the complement is $2v_3$. Note that $t_1=0, t_2=2, \Delta=4$. Then Theorem \ref{Thm Upper Bound} together with Equation (\ref{Eq with Delta}) in Section \ref{Section Triangular Regions} give $4 v_3$ in the bound. This can be compared, for example, with the bound from Theorem \ref{Lackenby:Volume} in terms of the number of twists, that gives $10v_3$. 
\end{example}

\section{Relation with coefficients of the colored Jones polynomial}
 
Let 
\begin{equation} \label{ColoredJones}
J_L(n)=\pm (a_n q^{k_n} - b_{n} q^{k_n-1} + c_{n} q^{k_n-2})+\ldots \pm (\gamma_n q^{k_n-r_n+2} - \beta_n q^{k_n-r_n+1} + \alpha_n q^{k_n-r_n})
\end{equation}
 be the colored Jones polynomial of an alternating link $K$, where $a_n$ and $\alpha_n$ are positive. 
 The absolute values of the first three and last three coefficients of the colored Jones polynomial are independent of the color $n$ when $n \geq 3$ (see \cite{DasbachLin:HeadAndTail}), and $b_n$ and $\beta_n$ are independent of $n$ for $n\geq 2$.
Moreover, the leading and trailing coefficients are known to satisfy $a_n=\alpha_n=1$ for all $n$.

We will use the values $b_2,c_2, c_3 , \beta_2, \gamma_2$ and $\gamma_3$ to give a refinement of the Volumish Theorem \cite{DL:VolumeIsh, DasbachLin:HeadAndTail} which states that for a prime, alternating, non-torus link $L$
the hyperbolic volume is bounded by:
$$\Vol(S^3-L) \leq 10 (b_2+\beta_2-1) v_3,$$
where, as before, $v_3$ is the volume of a regular ideal hyperbolic tetrahedron. 

\begin{figure}
\begin{tikzpicture}[scale=1]
\draw (-1,-1) -- (0,0) node[right=0.5cm]{B} node[left=0.5cm]{B} node[below=0.5cm]{W} node[above=0.5cm]{W} -- (1,1);
\draw (-1,1) -- (-0.08,0.08);
\draw (0.08,-0.08) -- (1,-1);

\end{tikzpicture}
\caption{Black (B) and white (W) shading of the alternating link diagram \label{BWcoloring}}
\end{figure}
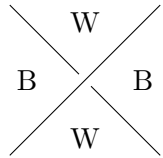

The values $b_2,c_2, c_3 , \beta_2, \gamma_2$ and $\gamma_3$ can be easily read off a reduced alternating diagram $D$ of an alternating link $L$ as follows.
Suppose $D$ is twist reduced in the sense of \cite{Lackenby:Volume}, i.e. the twist number $t(D)$ is minimized among all alternating diagrams of $L$.
Color the diagram $D$ in black and white as a checkerboard in such a way that black regions are clockwise from an over-strand, while white regions are counter-clockwise from an over-strand; see Figure \ref{BWcoloring}. One can associate a plane graph to the black regions by placing a vertex inside each black region, and by connecting two vertices if they are adjacent to the same crossing. We will call it the black checkerboard graph and denote it by $B(D)$. By using white regions instead, one obtains the planar white checkerboard graph $W(D)$. Note, that $B(D)$ and $W(D)$ are dual to each other. 
The vertices of $B(D)$ correspond to the faces of $W(D)$ and vice versa. The graphs $B(D)$ and $W(D)$ might have multiple edges between the same two vertices (such objects are often referred as multigraphs). If there are $m$ edges between two vertices $v_1$ and $v_2$ in a graph $G$, we leave just one of them (denote it by $e$) and delete the rest. The resulting graph is called the reduced graph, and we denote it by $G'$. Using the initial graph $G$, we can assign multiplicity to every edge $e$ of $G'$. Figure \ref{CheckerboardGraphs} shows a reduced alternating diagram of the knot $9_{20}$ together with its two reduced checkerboard graphs $B'(D)$ and $W'(D)$.

By $\tau_B$ or $\tau_W$ let us denote the number of triangles in  $B'(D)$ or $W'(D)$. 
Finally, denote by $n_B(i)$ or $n_W(i)$ the number of edges of multiplicity at least $i$ in $B'(D)$ or $W'(D)$. Thus, by definition, $n_B(i)+n_W(i)=g_i$.

\begin{figure}
\begin{subfigure}[b]{\textwidth}
\centering
\includegraphics[width=5cm]{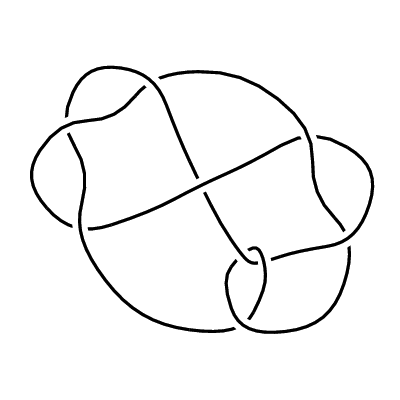}
\vspace*{-20pt}
\subcaption{A reduced alternating diagram $D$ of $9_{20}$}
\end{subfigure}

\begin{subfigure}[b]{0.34 \textwidth}
\vspace{20pt}
\centering
\begin{tikzpicture}[scale=1.7] 
\path (0,1) coordinate (X1); 
\path (-1,1) coordinate (X2);
\path (-1,0) coordinate (X3);
\path (0,0) coordinate (X4);
\path (1,0) coordinate (X5);
\path (1,1) coordinate (X6);
\path (0.5,0.5) coordinate (X7);
\foreach \j in {1, ..., 7} \fill (X\j) circle (1pt); 
\draw (X1)--(X2)--(X3)--(X4)--(X5)--(X6)--(X1)--(X4)--(X7)--(X5);
\end{tikzpicture}
\centering
\subcaption{Checkerboard graph without multiple edges $W(D)=W'(D)$}
\end{subfigure}
\begin{subfigure}[b]{0.34 \textwidth}
\centering
\begin{tikzpicture}[scale=1.7]
\path (0,0) coordinate (X1);
\path (1,0) coordinate (X2);
\path (0,1) coordinate (X3);
\path (-1,0) coordinate (X4);
\foreach \j in {1, ..., 4} \fill (X\j) circle (1pt); 
\draw (X1)--(X2)--(X3)--(X4)--(X1)--(X3);
\end{tikzpicture}
\subcaption{Reduced checkerboard graph $B'(D)$}
\end{subfigure}
 
\caption{The knot $9_{20}$ and its two reduced checkerboard graphs. \label{CheckerboardGraphs}}
\end{figure}

The following proposition follows immediately from results in \cite{DasbachLin:HeadAndTail}:

\begin{proposition} \label{lem:threecoeff}
Let the colored Jones polynomial be given in the form as in Equation (\ref{ColoredJones}).
Then
\begin{eqnarray*}
b_2+\beta_2 &=& t(D) \\
&=& t_1(D)+g_2(D)\\
(c_2+\gamma_2)-(c_3+\gamma_3)&=& t(D)+g_2(D)\\
&=&t_1(D)+2 g_2(D)\\
(c_2+c_3)+(\gamma_2+\gamma_3)-(b_2^2+\beta_2^2)&=&g_2 -2 (\tau_A(D)+\tau_B(D))
\end{eqnarray*}
\end{proposition} 

\begin{example}
Let $K$ be the knot $9_{20}$ with a diagram $D$ as in Figure \ref{CheckerboardGraphs}.
There are $5$ twists, i.e. $t(D)=5$, and $t_1(D)=2, t_2(D)=2, t_3(D)=1$ and $\tau_A(D)=1, \tau_B(D)=2$.
The colored Jones polynomial of $K$ is given by
\begin{eqnarray*}
J_K(2)&=&-\frac{1}{q^9}+\frac{3}{q^8}-\frac{5}{q^7}+\frac{6}{q^6}-\frac{7}{q^5}+\frac{7}{q^4}-\frac{5}{q^3}+\frac{4}{q^2}-\frac{2}{q}+1\\
J_K(3)&=& \frac{1}{q^{25}}-\frac{3}{q^{24}}+\frac{2}{q^{23}}+ \ldots -1-2 q + q^2
\end{eqnarray*} 
Hence, $b_2=2, c_2=4, c_3=-1, \beta_2=3, \gamma_2=5,$ and $\gamma_3=2.$
\end{example}

Combining Proposition \ref{lem:threecoeff} with Theorem \ref{Thm Upper Bound} we obtain

\begin{theorem} 
\label{refinedVolumeIsh}
Let $K$ be an alternating, prime, non-torus link. Then
$$\Vol(S^3-K) \leq \left ( 6 ((c_2+\gamma_2) - (c_3 + \gamma_3)) - 2 (b_2+ \beta_2) -a \right ) v_3 \leq 10 (b_2+\beta_2-1) v_3,$$
where $a=10$ if $b_2+\beta_2 \neq (c_2-c_3)+(\gamma_2-\gamma_3)$ and $a=4$ otherwise.
\end{theorem} 

\begin{remark}
It is tempting to try to bound the $\Delta$ in Section \ref{Section Triangular Regions} from below by $\tau_A$ and $\tau_B$. However, $\tau_A$ and $\tau_B$ count triangles in the checkerboard graphs, and not triangular regions in the planar embeddings of the checkerboard graphs.
\end{remark}

\section{Acknowledgments }
 
The authors are grateful to Jessica Purcell for many enlightening conversations and interest in this work. We would also like to thank Colin Adams, Ian Agol, Dave Futer, Effie Kalfagianni, Marc Lackenby and Peter Shalen for helpful discussions on various occasions, and to the anonymous reviewer for aid in improving the paper. The authors acknowledge support from U.S. National Science Foundation grants DMS-1317942, DMS-0739382 (VIGRE), and DMS-1406588.
  
\bibliography{WithAnastasiia}
\bibliographystyle {amsalpha}

\parbox[t]{2.8in}{
Oliver Dasbach\\
Department of Mathematics\\
Louisiana State University\\
Baton Rouge, Louisiana 70803, USA.\\
kasten@math.lsu.edu}
\qquad
\parbox[t]{2.8in}{
Anastasiia Tsvietkova\\
Department of Mathematics\\
University of California, Davis \\
One Shields Ave, Davis, CA 95616\\
tsvietkova@math.ucdavis.edu
}

\end{document}